\documentclass[12pt]{article}
 
\usepackage{amssymb}
\usepackage{amsmath}
\usepackage{amsthm}

\usepackage{hyperref}
\usepackage{url}
\usepackage{mathrsfs}
\usepackage{pdfpages}
\newtheorem{thm}{Theorem}[section]

\newtheorem{lem}[thm]{Lemma}

\newtheorem{theor}[thm]{Theorem}
\newtheorem{cor}[thm]{Corolary}

\newcommand{\Me}{\mathcal{M}}  
\newcommand{\St}{\mathcal{S}}

\newcommand{\Obs}{\mathrm{Obs}}
\newcommand{\distr}{\mathrm{distr}}
\newcommand{\MALG}{\mathrm{MALG}}

\newcommand{\N}{\mathbb{N}}
\newcommand{\sA}{\mathscr{A}} 

\newcommand{\So}{\mathfrak{S}}

\newcommand{\Sym}{\mathrm{Sym}}

\title{Weak containment and maximal sofic approximations}
\author{Andrei Alpeev\footnote{Chebyshev Laboratory, St. Petersburg State University, 14th Line V.O., 29B, Saint Petersburg 199178 Russia,  \href{mailto:a.alpeev@spbu.ru}{a.alpeev@spbu.ru}}   }

\begin{document}

\maketitle
\begin{abstract}
We show that the class of sofic actions is closed under direct products and contains a (non-unique) maximal element in the weak containment order. For any sofic group we construct nice sofic approximations such that all the sofic actions are approximable by them in a doubly-quenched way. We use recent result by Hayes to establish for these sofic approximations the equality of sofic measure entropy to the topological one for algebraic actions whenever the former is not $-\infty$. We also use his another result to establish the product formula for Pinsker factors of these approximations.
\end{abstract}
keywords: sofic action, sofic entropy, weak containment, algebraic action, Pinsker factor
\newpage
\tableofcontents

\section{Introduction}

The notion of weak containment was introduced by Kechris in the book \cite{K10} in the context of cost theory of Borel eqiuvalence relations. It turned out to be a very fruitfull concept. It is known that this pre-order has a maximal element(see \cite{GTW06} and \cite{K10}). It was proved by Abert and Weiss \cite{AbW13} that Bernoulli actions are minimal in the class of free actions. In the work \cite{S16} Seward introduced the notion of weak containment for joinings and used it in the context of so-called  Rokhlin entropy.
Another notion we will work with is the notion of sofic action. In the groundbreaking paper \cite{B10} Lewis Bowen have defined so-called sofic entropy, the notion realiant on mimicking of a given acion over so-called  sofic approximation(which is a finite approximation of a group).
An action which could be approximated in such a way is called sofic action. 
It was already known that weak containent has something to do with soficity. It is a folklore fact that an action weakly contained in the sofic action is sofic itself. Carderi in \cite{C15} connected this notion to the soficity of action relative to given sofic approximation, namely he showed that action is approximable by given ultralimit sofic approximation iff it is weakly contained in the ultraproduct action of this sofic approximation. In the work \cite{BTD17} Bowen an Tucker-Drob have studied the space of so-called stable weak equivalence classes.

In this work we will firstly show that product of two sofic actions is also sofic (theorem \ref{theor: product is sofic}). We then show that there is a (non-unique) maximal element in the class of sofic actions of given sofic group with respect to the weak containment order (theorem \ref{theor: universal sofic action}) we will call it a {\em universal sofic action}. We then define a {\em maximal sofic approximation} to be any sofic approximation which approximates some (and hence any) universal action.
In the theorem \ref{theor: maximal sofic approximation} we show that any sofic action will be aprroximable by it in a nice fashion (without need to taking subsequences). Afterwards we recall the notion of doubly-quenched convergence introduced by Austin in \cite{Au16}. In the theorem \ref{theor: all sofic are strongly sofic} we prove using result of Austin that any sofic action admits a doubly-quenched approximation, the situation called by Hayes ``strong soficity'', with respect to maximal sofic approximation. In works \cite{H16a}, \cite{H16b} he showed that strong soficity of an action with respect to given sofic approximation has a lot of interesting consequences. 
In the section \ref{sec: applications} we show how our results are combining with theorems of Hayes from \cite{H16a}, \cite{H16b}. It is a an important problem to understand whether for algebraic actions topological sofic entropy equals to the measure-theoretic one (then we endow this action with the Haar measure). In the corolary \ref{cor: algebraic} we use the result from \cite{H16a} to show that for maximal sofic approximations this equality holds whenever the measure-theoretic entropy is not $-\infty$. We also show in corolary \ref{cor: pinsker} that a product formula for Pinsker factors holds for maximal sofic approximations .

We study the class $\St$ of sofic groups all of whose actions are soficin the section \ref{sec: class s}. By a simple argument and using well-known result that all the treable actions are sofic (see \cite{EL10},\cite{Pa11}),  we show that all the treeable groups (groups admitting an essentially free treeable action) belong to this class. We also show that this class is closed under taking subgroups.

{\em Acknowledgements. } I would like to thank Ben Hayes for his comments. Research is supported by the Russian Science Foundation grant №14-21-00035.

\section{Measures and Kantorovich distance}
Let $G$ be a countable group. We will fix an arbitrary enumeration of its elements $(\gamma_i)_{i \in N}$. 

For a standard Borel space $X$ we will denote $\Me(X)$ the set of all the Borel probability measures on $X$. 

For a probability space $(X,\mu)$ let us denote $\MALG(X,\mu)$ the space of all the measurable subsets endowed with the symmetric difference distance.

Let $(X, r)$ be a compact metric space. We denote $l$ --- the {\em Kantorovich distance} on $\Me(X)$ defined by
\[
l(\mu_1,\mu_2) = \inf_\xi \int_{X \times X} r(x_1,x_2) d\xi(x_1,x_2), 
\]
there infimum is taken along all the couplings $\xi$ of measures $\mu_1$ and $\mu_2$.

It is known that the Kantorovich distance defines the weak* topology on the space $\Me(X)$.

Let $f$ be a map between two metric compacta $X$ and $Y$. If $f$ does not increase distance then the induced map on measures does not increase the Kantorovich distance.

A convention on metrics.
\begin{enumerate}
\item If metric for the space is explicitly stated then this is it.
\item If the space has the form $\Me(X)$ for $X$ a compact set with specified metric then we endow it with the Kantorovich distance.
\item If the space has the form $X^V$ there $V$ is a finite set and metric for $X$ is specified somehow then we will define metric on $X^V$ by
\[
d_{X^V}(t_1,t_2) = \sup_{v \in V} d_X(t_1(v),t_2(v))
\]
for $t_1,t_2 \in X^V$.
\item If the space has the form $X^G$ and for $X$ the metric is specified then we define the metric by
\[
d_{X^G} = \sup_{i \in \N} \frac1i d_X (t_1(\gamma_i),t_2(\gamma_i))
\]
for $t_1,t_2 \in X^G$.
\item All the finite sets we will endow with the discrete metric ($1$ for distinctive element and $0$ otherwise).
\end{enumerate}
\section{Obseravables}
All the actions of countable groups on probability spaces will be measure preserving if otherwise is not stated.

Let us fix a measure preserving action of countable group $G$ on the standard probability space $(X,\mu)$. An {\em observable} is any measrable map $f : X \to A $ where $A$ is a finite space. We denote $\Obs(X,\mu,A)$ the space of observables from $X$ to $A$. 
We will endow it with the metric: for $f_1, f_2 \in \Obs(X,\mu,A)$ it will be defined as $\mu(\lbrace x \vert f_1(x) \neq f_2(x)\rbrace)$. 
For observable $f$ we will define new map $f^G: X \to A^G$ by $(f^G(x))(g)= f(g x)$ for $x \in X$ and $g \in G$.
Consider a map $\distr: \Obs(X, \mu,A) \to \Me(A^G)$ defined as $f \mapsto f^{G}(\mu)$. We will also denote the action if needed: $\distr_T (f)$.

\begin{lem}\label{lem: distr continuous}
The map $\distr$ is continuous.  
\end{lem}
\begin{proof}
It is enough to prove that measure of any cylinder set depends continuously on $f$. We can see that for any cylinder subset $C$ of $A^G$ the map $\Obs(X,\mu,A) \to \MALG(X,\mu) $ defined by $f \mapsto {(f^G)}^{-1}(C)$ is continuous. This implies the desired.
\end{proof}

Let $V$ be any set. Let $A$, $B$ be any finite sets. Let $\pi$ be a map from $B$ to $A$. We will denote $\pi^V: B^V \to A^V$ to be the map defined naturally by applying $\pi$ coordinate-wise. We will define a {\em tower of observables} to be the sequence of observables $f_i : X \to A_i$ together with the maps $\pi_{i,j}: A_i \to A_j$ such that for any $i<j$ holds $(\pi_{i,j} \circ f_i)(x) = f_j(x)$ for $\mu$-a.e. $x \in X$ and that for any $i<j<k$ we have $\pi_{j,i} \circ \pi_{k,j} = \pi_{k,i}$. We will say that the tower is {\em sufficient} if for any observable $f: X \to B$ and $\varepsilon>0$ there is an index $i$ and a map $\pi': A_i \to B$ such that observable $\pi' \circ f_i$ is $\varepsilon$-close to $f$. We note that it is easy to construct a sufficient tower of observables on the standard probability space.

We will say that p.m.p. action of $G$ on $(X_1,\mu_1)$ is {\em weakly contained} in the p.m.p. action of $G$ on $(X_2,\mu_2)$ if for any observable $f_1: X_1 \to A $ and any $\varepsilon>0$ there is an observable $f_2 : X_2 \to A$ such that $\distr(f_2)$ is $\varepsilon$-close to $\distr(f_2)$.

\section{Sofic groups and sofic actions}
For a finite set $V$ we will denote $\Sym(V)$ the group of all its permutations. We will endow it with the normalised Hamming distance:
\[
d_H(g_1,g_2) = \frac{ \lvert \lbrace v \in V \;,\; g_1(v) \neq g_2(v)\rbrace\rvert }{\lvert V\rvert} 
\]

Let $G$ be a countable group. A {\em sofic approximation} $\So$ is a sequence of finite sets $(V_i)$ together with a sequence of maps $(\sigma_i)$, $\sigma_i : G \to \Sym(V_i)$ (we will use the notation $\sigma_i : g \mapsto \sigma_i^g$) satisfying two properties:
\begin{enumerate}
\item $\lim_{i \to \infty} d_H(\sigma_i^{g_1},\sigma_i^{g_2}) = 1$ for any $g_1 \neq g_2$ from $G$,
\item $\lim_{i \to \infty} d_H(\sigma_i^{g_1} \circ \sigma_i^{g_2}, \sigma_i^{g_1 g_2}) = 0$ for any $g_1,g_2$ from $G$.
\end{enumerate}

A group is said to be {\em sofic} if it has at least one sofic approximation. 

Let $\sigma$ be a map from $G$ to $\Sym(V)$ for some finite set $V$. Let $X$ be any set. We define a map $\theta_{v, \sigma}: X^V \to X^G$ by equality
\[
(\theta_{v,\sigma}(\tau))(g) = \tau(\sigma^g(v)).
\]
We will denote $\Theta_{v,\sigma}$ the map from $\Me(X^V)$ to $\Me(X^G)$ defined by equality
\[
\Theta_{\sigma}(\eta) = \frac1{\lvert V \rvert}\sum_{v \in V} \theta_{v,\sigma}(\eta).
\]
We will omit $\sigma$ in the notations above if it is clear from the context.

Let $\So = ((V_i),(\sigma_i))$ be a sofic approximation. Let $G$ be acting on the the standard probability space $(X,\mu)$. We will say that this action is $\So$-weakly approximable if for any observable $f: X \to A $ there is a growing sequence $(i_j)$ of natural numbers and a sequence $(\tau_j)$, $\tau_j \in A^{V_{i_j}}$ such that 
\[
\lim_{j \to \infty} \Theta(\delta_{\tau_{i_j}}) = \distr(f).
\]  
It is not hard to see that 
We will say that this action is $\So$-approximable if for any observable $f: X \to A $ there is a  a sequence $(\tau_i)$, $\tau_i \in A^{V_i}$ such that 
\[
\lim_{i \to \infty} \Theta(\delta_{\tau_i}) = \distr(f).
\] 
We will say that action is {\em sofic} if it is $\So$-weakly aproximable for some sofic approximation $\So$ of the group $G$.

For fixed sofic group $G$ we will say that the map $\sigma: G \to \Sym(V)$ is $k$-good if
\begin{enumerate}
\item $ d_H(\sigma_i^{\gamma_i},\sigma_i^{\gamma_j}) > 1 - 1/k$ for any $i < j \leq k $,
\item $ d_H(\sigma_i^{\gamma_i} \circ \sigma_i^{\gamma_j}, \sigma_i^{\gamma_i \gamma_j}) < 1/k$ for any $i,j \leq k$.
\end{enumerate}

It is not hard to see that we can reformulate now the requirements from the definition of sofic approximation in the following way: for any natural $k$ there is such $i_0$ that for any $i>i_0$ we have that $\sigma_i$ is $k$-good.

It is easy to see that an action $G \curvearrowright (X,\mu)$ is sofic iff for any $f:X \to A$ and for any natural $k$ there is a $k$-good map $G \to \Sym(V)$ for some finite $V$ and an element $\tau \in A^V$ with $\Theta_{\sigma}(\delta_{\tau})$ being $1/k$-close to $\distr(f)$.

Let $(f_i)$, $f_i : X \to A_i$ be a sufficient tower of observables on $(X,\mu)$. Then an action $T$ of $G$ on $(X,\mu)$ is sofic iff for any natural $k$ we have that there is $\tau \in A_k^V$ and $\sigma: G \to \Sym(V)$ --- such a $k$-good map(for some finite $V$) that $\Theta_{\sigma}(\delta_\tau)$ is $1/k$-close to $\distr(f_k)$.

\begin{lem}\label{lem: sofic approximation for sofic action}
Any sofic action $T$ of sofic group $G$ on a standard probability space $(X,\mu)$ is $\So$-approximable for some sofic approximation $\So$. 
\end{lem}
\begin{proof}
We fix a sufficient tower of observables $(f_i)$, $f_i : X \to A_i$ and take $\sigma_i: G \to \Sym(V_i)$ such that it is $i$-good and that there is such a $\tau \in A_i^{V_i}$ that $\Theta_{\sigma_i}(\delta_\tau)$ is $1/i$-close to $\distr(f_i)$. It is not hard to see now that for any natural $j<i$ we have that $\Theta_{\sigma_i}(\delta_{\pi_{i,j}^{V_i}(\tau)})$ is $1/i$-close to $\distr(f_j)$. Hence our action is approximable with respect to constructed sofic approximation.
\end{proof}
\begin{theor}\label{theor: product is sofic}
Direct product of two sofic actions of sofic group on standard probability spaces is a sofic action.
\end{theor} 
\begin{proof}
Let $T',T''$,  be two sofic actions of $G$ on standard probability spaces $(X',\mu')$,$(X',\mu')$ respectively. Let $(f_i')$, $f'_i : X' \to A_i'$ be a sufficient tower of observables for $(X',\mu')$ and $(f''_i)$, $f''_i : X' \to A_i'$ --- on $(X'',\mu'')$. Let us consider a tower of observables $(f'_i \otimes f''_i)$ on $(X' \times X'', \mu' \otimes \mu'')$ defined by 
\[
f'_i \otimes f''_i : (x',x'') \mapsto (f'_i(x'),f''_i(x'')).
\]
It is not hard to see that this tower of observables is sufficient.
Suppose $T'$ is $\So'$-approximable and $T''$ is $\So''$-approximable for some sofic approximations $\So',\So''$. For a natural $j$ let $(\tau'_i), \tau'_i \in {A'_j}^{V_i}$ be such a sequence that 
\[
\lim_{i \to \infty} \Theta_{\sigma'_i}(\delta_{\tau'_i}) = \distr(f'_j).
\] 
Analogously we take $(\tau''_i), \tau''_i \in {A''_j}^{V_i}$ be such a sequence that 
\[
\lim_{i \to \infty} \Theta_{\sigma''_i}(\delta_{\tau''_i}) = \distr(f''_j).
\] 
We now consider the sequence $(\tau'_i \otimes \tau''_i)$, $\tau'_i \otimes \tau''_i : V'_i \times V''_i \to A'_j \times A''_j$ defined by 
\[
(\tau'_i \otimes \tau''_i)(v',v'') = (\tau'_i(v'), \tau''_i(v'')).
\]
We define a sofic approximation $\So' \times \So''$ in such a way that $\sigma'_i \times \sigma''_i: G \to \Sym(V'_i \times V''_i)$ and 
\[
(\sigma'_i \times \sigma''_i)^g(v',v'') = ({\sigma'}^g_i(v'),{\sigma''}^g_i(v''))
\]
for any $g \in G$, $v' \in V'_i$ and $v'' \in V''_i$.
It is a sofic approximation indeed and we can check that
\[
\lim_{i \to \infty} \Theta_{\sigma'_i \times \sigma''_i}(\delta_{\tau'_i} \otimes \delta_{\tau''_i}) = \lim_{i \to \infty} \Theta_{\sigma'_i \times \sigma''_i}(\delta_{\tau'_i \otimes \tau''_i}) = \distr(f'_i \otimes f''_i).
\]
This implies that $T' \times T'$ is $\So' \times \So''$-approximable and hence sofic.
\end{proof}

\begin{lem}
Inverse limit of sofic actions is sofic. 
\end{lem}
\begin{proof}
Let $T$ be an action of $G$ on a standard probability space $(X,\mu)$, let $(\sA_i)$ be a an increasing sequence of $G$-invariant subalgebras whose join is the whole algebra of measurable sets on $X$. Suppose that all the corresponding factors are sofic. 
Fix any observable $f: X \to A$ and a natural $k$. We can find such a close observable $f': X \to A$ that it is $\sA_i$-measurable for some $i$ and that $\distr(f')$ is $1/(2k)$-close to $\distr(f)$ (by lemma \ref{lem: distr continuous}).  
Since factor corresponding to $\sA_i$ is sofic, we can find a $k$-good map $\sigma : G \to \Sym(V)$ for some finite $V$ and $\tau \in A^V$ such that $\Theta_{\sigma}(\delta_{\tau})$ is $1/(2k)$-close to $\distr(f')$. We have now that $\Theta_{\sigma}(\delta_{\tau})$ is $1/k$-close to $\distr(f)$. This implies that $T$ is a sofic action.  
\end{proof}

\begin{lem}\label{lem: countable products}
Countable product of sofic actions is sofic.
\end{lem}
\begin{proof}
It's a direct consequence of two previous lemmata.
\end{proof}

We are now ready to prove that for any countable group there is a universal sofic action which weakly contains any other.
\begin{theor}\label{theor: universal sofic action}
For any sofic group $G$ there is such a sofic action $T$ that any other sofic action is weakly contained in it.
\end{theor}
\begin{proof}
The proof is completely analogous to that of the fact stating the existence of weakly maximal action among all the actions of the given group (see \cite{GTW06}, \cite{K10} and \cite{BuK16}. Of course we will need the previous lemma.
For every natural $i$ let $A_i$ be a set of size $i$. We now consider the set $D_i$ of all such measures $\mu$ on $A_i^G$ that there is a sofic action $T$ on some standard probability space $(X,\mu)$ and an observable $f: X \to A_i$ that $\mu= \distr_T(f)$. We now pick a countable dense set of measures from this set and for each measure take a corresponding action. We now proceed like this for all the natural $i$ which gives us a countable collection of sofic actions. Their product will be sofic by the previous lemma. It is easy to check that any sofic approximation is weakly contained in this product.
\end{proof}

We will call any such action from the previous theorem a {\em universal sofic action}. Sofic approximation $\So$ is called {\em maximal} if any universal sofic action is $\So$-approximable.

We now show that maximal sofic approximation approximates any sofic action.
\begin{theor}\label{theor: maximal sofic approximation}
If $\So$ is a maximal sofic approximation for sofic group $G$ then any sofic action is $\So$-approximable. 
\end{theor}
\begin{proof}
Any sofic action is weakly contained in a universal sofic approximation which is $\So$-approximable. Hence any action is $\So$-approximable.
\end{proof}
We note that our definition of $\So$-approximable action is rather strong and does not allow for taking subsequences to approximate.

We will say that an action $T$ of group $G$ on $(X,\mu)$ is $\So$-{\em strongly sofic} if for any observable $f: X \to A$ there is a sequence $(\eta_i)$, $\eta_i \in \Me(A^{V_i})$ such that 
\[
\lim_{i \to \infty}\Theta(\eta_i \otimes \eta_i) = \distr(f) \otimes \distr(f)
\]
and for any $\varepsilon>0$ the sets 
\[
W_{i,\varepsilon} = \lbrace \xi \in (A \times A)^V_i  \quad\vert\quad l(\Theta(\delta_{\xi}, \distr(f) \otimes \distr(f)) <\varepsilon\rbrace
\]
satisfy 
\[
\lim_{i \to \infty} \eta_i \otimes \eta_i(W_{i,\varepsilon}) = 1.
\]

This kind of convergence condition is called {\em doubly-quenched convergence} and it was introduced by Austin in \cite{Au16}. We will need the following theorem from his work. For action $T$ and a natural $n$ we will denote $T^n$ the product-action of the same group.  
\begin{theor}[Austin, \cite{Au16}]
Let $\So$ be a sofic approximation and $T$ be an action. If for any  natural $n$ we have that $T^n$ is $\So$-approximable then $T$ is strongly sofic with respect to $\So$.
\end{theor}

\begin{theor}\label{theor: all sofic are strongly sofic}
Let $\So$ be a maximal sofic approximation for a sofic group $G$. Then any sofic action of $G$ is strongly sofic with respect to $\So$.
\end{theor} 

\begin{proof}
We note that if action $T$ is sofic then $T^n$ is sofic for any natural $n$ by the theorem \ref{theor: product is sofic}. It implies that $T^n$ is $\So$-approximable for any natural $n$. Hence it is $\So$-stronly sofic by the theorem of Austin.
\end{proof}


\section{Class $\St$}\label{sec: class s}

The class $\St$ was defined in by Paunescu in \cite{Pa11} to be the class of all the sofic groups such that all the m.p.t. actions of these groups are sofic. He also discovered some properties of this class. 

Let $T$ be an m.p.t. action of group $G$ on the standard probability space $(X,\mu)$. Consider correspondent orbit equivalence relation $\sim$. Consider any collection $(\psi_i)$ of partial Borel maps (i.e. defined on a Borel subset) from $X$ to itself such that their graphs are subsets of an equivalence relation $\sim$. This sequence will be called {\em graphing} if induced graph is such that for almos every $x \in X$ and every $y \sim x$ we have that $x$ and $y$ are connected via the graph induced by the collection $(\psi_i)$. This graphing is called {\em treeing} if this graph is acyclic. An action relation is called {\em treeable} if its equivalence relations admits a treeing. 
A countable group group is called {\em treeable} if it admits an essentially free treeable action.

\begin{theor}
Class of treeable groups is subset of $\St$.
\end{theor}
\begin{proof}
Take any actions $T$. Take any essentially free treeable action $T'$. Their direct product will be essentially free and treeable. Hence it is sofic (see \cite{EL10} and \cite {Pa11}). So $T$ is a factor of sofic action, hence it is sofic.
\end{proof}

\begin{theor}
Class $\St$ is closed under taking subgroups.
\end{theor}
\begin{proof}
Let $H$ be a subgroup of $G \in \St$. Let $T$ be an action of $H$. Consider a co-induced action $\tilde{T}$ of $G$. It is sofic. We note that induced action $\bar{T}$ of $H$ will be sofic and it is isomorphic to the product of $\lvert G : H\rvert$ copies of $T$. Hence $T$ is a factor of sofic action and hence it is sofic.
\end{proof}

\section{Applications}\label{sec: applications}

Let $X$ be a metrizable compact topological group and let $T$ be an action of countable group $G$ on $X$ by continuous automorphisms. Such an action is called an {\em algebraic action}. It could be endowed with the Haar measure, it is easy to see that this measure is presserved by this action. We refer the reader to \cite{H16a} for the definition of sofic topological and measure entropy.

\begin{cor}\label{cor: algebraic}
Let us fix a maximal sofic approximation. Then for any algebraic action we will have that  measure sofic entropy of this action endowed with the Haar measure is either $-\infty$ or equals to the topological entropy. If acting group belongs to $\St$ then we have equality for maximal sofic approximations without additional conditions.
\end{cor} 
\begin{proof}
If measure entropy is not $-\infty$ then we have that algebraic action endowed with the Haar measure is sofic (we note that if acting group belongs to $\St$ then this is always the case).
Hence it is strongly sofic with respect to our approximation (theorem \ref{theor: all sofic are strongly sofic}). By the theorem 1.1 of \cite{H16a} its measure entropy equals to the topological entropy.
\end{proof}

We refer the reader to \cite{H16a} for the definitiojn of Pinsker factor for sofic entropy in the presence.

\begin{cor}\label{cor: pinsker}
Let $T_i$ act on standard probability space $(X_i,\mu_i)$ for $i=1,2$. Let $\So$ be a maximal sofic approximation and suppose that $T_i$ are sofic. Then Pinsker factors for sofic entropy in the presence for considered actions obey the prodict formula: Pinsker subalgebra for product action equals to the join of Pinsker algebras for $T_i$, $i=1,2$.
\end{cor}
\begin{proof}
Our actions are strongly sofic with respect to $\So$ by \ref{theor: all sofic are strongly sofic} and the statement follows by the theorem 1.1 from \cite{H16b}.
\end{proof}


\begin{thebibliography}{100000}

\bibitem[AbW13]{AbW13} M. Abert and B. Weiss, \textit{Bernoulli actions are weakly contained in any free action}, Ergodic theory and dynamical systems 33.02 (2013): 323-333.
APA	





\bibitem[Au16]{Au16} T. Austin \textit{Additivity properties of sofic entropy and measures on model spaces}, Forum of Mathematics, Sigma. Vol. 4. Cambridge University Press, 2016.


\bibitem[B10]{B10} L. Bowen, \textit{Measure conjugacy invariants for actions of countable sofic groups}, Journal of the American Mathematical Society 23 (2010): 217--245.



\bibitem[B17]{B17} L. Bowen, \textit{Examples in the entropy theory of countable group actions}, arXiv preprint arXiv:1704.06349 (2017).

\bibitem[BTD17]{BTD17} L. Bowen, Lewis and R. Tucker-Drob, \textit{The space of stable weak equivalence classes of measure-preserving actions}, arXiv preprint arXiv:1705.03528 (2017).



\bibitem[BuK16]{BuK16} P. Burton and A. S. Kechris, \textit{Weak containment of measure preserving group actions}, arXiv preprint arXiv:1611.07921 (2016).



\bibitem[C15]{C15} A. Carderi, \textit{Ultraproducts, weak equivalence and sofic entropy}, arXiv preprint arXiv:1509.03189 (2015).


\bibitem[EW11]{EW11} M. Einsiedler and T. Ward,  \textit{Ergodic theory with a view towards number theory}. Graduate texts in mathematics, 259. Springer, London, 2011.

\bibitem[EL10]{EL10} G. Elek and G. Lippner, \textit{Sofic equivalence relations}, Journal of Functional Analysis 258.5 (2010): 1692-1708.
APA	

\bibitem[GTW06]{GTW06} E. Glasner,  J-P. Thouvenot, and B. Weiss, \textit{Every countable group has the weak Rohlin property}, Bulletin of the London Mathematical Society 38.6 (2006): 932-936.


\bibitem[H14]{H14} B. Hayes, \textit{Fuglede-kadison determinants and sofic entropy}, arXiv preprint arXiv:1402.1135 (2014).

\bibitem[H16a]{H16a} B. Hayes, \textit{Doubly Quenched Convergence and the Entropy of Algebraic Actions of Sofic Groups}, arXiv preprint arXiv:1603.06450 (2016).

\bibitem[H16b]{H16b} B. Hayes, \textit{Relative entropy and the Pinsker product formula for sofic groups}, arXiv preprint arXiv:1605.01747 (2016).



\bibitem[K10]{K10} A.S. Kechris,\textit{Global aspects of ergodic group actions}, Vol. 160. Providence, RI: American Mathematical Society, 2010.

\bibitem[KM04]{KM04} A.S. Kechris and B. D. Miller, \textit{Topics in orbit equivalence}, No. 1852. Springer Science \& Business Media, 2004.

\bibitem[Ke13]{Ke13} D. Kerr, \textit{Sofic measure entropy via finite partitions}, Groups Geom. Dyn, 7(2013): 617--632.


\bibitem[Pa11]{Pa11} L. Paunescu, \textit{On sofic actions and equivalence relations}, Journal of Functional Analysis 261.9 (2011): 2461--2485.




\bibitem[S16]{S16} B. Seward, \textit{Weak containment and Rokhlin entropy} arXiv preprint arXiv:1602.06680 (2016).

\end{thebibliography}
\end{document}